\documentclass[11pt]{amsart}
\usepackage[all]{xy}
\usepackage{amsmath}
\usepackage{amsfonts}
\usepackage{amssymb}
\usepackage{amscd}
\usepackage{amsthm}
\usepackage{latexsym}
\usepackage{amsbsy}
\usepackage{color,enumerate}

\def\0D{\Delta^{(0)}}
\def\1D{\Delta^{(1)}}

\newtheorem{theorem}{Theorem}[section]
\newtheorem{remark}[theorem]{Remark}
\newtheorem{proposition}[theorem]{Proposition}
\newtheorem{lemma}[theorem]{Lemma}

\newtheorem{example}[theorem]{Example}
\newtheorem{definition}[theorem]{Definition}

\def\build#1_#2^#3{\mathrel{\mathop{\kern 0pt#1}\limits_{#2}^{#3}}}

\def\odots{\ot\cdots\ot}

\numberwithin{equation}{section}
\def\a{\alpha}
\def\b{\beta}

\def\d{\delta}

\def\vp{\varphi}

\def\ot{\otimes}
\def\part{\partial}

\def\text{\hbox}

\def\ot{\otimes}

\def\Hom{\mathop{\rm Hom}\nolimits}

\def\Id{\mathop{\rm Id}\nolimits}

\def\build#1_#2^#3{\mathrel{
\mathop{\kern 0pt#1}\limits_{#2}^{#3}}}

\numberwithin{equation}{section}
\newcommand{\comment}[1]{\relax}

\begin{document}
%\Large
\title{Hom-Groups, Representations and Homological Algebra}
\author { Mohammad Hassanzadeh }
%\date{University of Windsor, Canada
%}
\curraddr{University of Windsor, Department of Mathematics and Statistics, Lambton Tower, Ontario, Canada.
}
\email{mhassan@uwindsor.ca}
\subjclass[2010]{ 17D99, 06B15, 20J05}
 \keywords{ Nonassociative rings and algebras, Representation theory,   Homological methods in group theory}
\maketitle
\begin{abstract}

A Hom-group $G$ is a nonassociative version of a group where associativity, invertibility, and unitality are twisted by a map $\alpha: G\longrightarrow G$.
Introducing the Hom-group algebra $\mathbb{K}G$, we observe that Hom-groups are providing examples of Hom-algebras, Hom-Lie algebras and Hom-Hopf algebras.
 We introduce two types of modules over  a Hom-group $G$. To find out more about these modules,
we  introduce Hom-group (co)homology with coefficients in these modules. Our (co)homology theories generalizes group (co)homologies for  groups.
Despite the associative case we observe that the  coefficients of  Hom-group homology is different from the ones for Hom-group cohomology.
 We show  that the inverse elements provide  a relation  between Hom-group (co)homology  with coefficients in right and left $G$-modules.
It will be shown that  our (co)homology theories for Hom-groups with coefficients could be reduced to the Hochschild (co)homologies of  Hom-group algebras.
For certain  coefficients the functoriality of  Hom-group (co)homology  will be shown.

\end{abstract}

%%%%%%%%%%%%%%%%%%%%%%%%%%%%%%%%%%%%%%%%%%%%%%%%%%%%%%%%%%%%%%%%%%%%%%%%%%%%%%%%%%%%%%%%%%%%%%%%%%%%%%%%%%%%%%%%%%%%%%%%%%%%%%%%%%%%%%%%%%%%%%%%%%%%%%%%%%%%%%%%%%%%%%%%%%%%%%%%%%%%%%%%%%%%%%%%%%%%%%%%%%%%%%%%%%%%%%%%%%%%%%%%%%%%%%%%%%%%%%%%%%%%%%%%%
\section{ Introduction}

The notion of Hom-Lie algebra is a generalization of  Lie algebras which  appeared
first in $q$-deformations of  Witt and Virasoro algebras  where  the Jacobi identity is deformed
by a linear map \cite{as}, \cite{ckl}, \cite{ cz}. There are several interesting examples of Hom-Lie algebras. For an example
the authors in \cite{gr} have shown that any  algebra of dimension 3  is a Hom-Lie algebra.
The related algebra structure is called Hom-algebra   and introduced in \cite{ms1}.
 Later the other objects such as  Hom-bialgebras and Hom-Hopf algebras were studied in \cite{ms2}, \cite{ms3}, \cite{ ya2}, \cite{ya3}, \cite{ya4}.
 We refer the reader to  more work for
 Hom-Lie algebras to \cite{cs}, \cite{hls}, \cite{ls}, \cite{bm}, for
  Hom-algebras to  \cite{gmmp}, \cite{fg}, \cite{hms},  and for  representations of Hom-objects to \cite{cq}, \cite{gw},  \cite{pss}.
One knows that studying Hopf algebras have close relations to groups and Lie algebras. The set of group-like elements and primitive elements  of a
 Hopf algebra form a group and a Lie algebra respectively. Conversely any group gives a Hopf algebra which  is called group algebra. For
  any Lie algebra we have universal  enveloping algebra. There have been many work relating Hom-Lie algebras and Hom-Hopf algebras.
  However some relations were missing in  the context of Hom-type objects due to the Lack of  Hom-type of notions for groups and group algebras.
  Here we briefly explain how Hom-groups were came in to the context of Hom-type objects.
The universal enveloping algebra of a Hom-Lie algebra has a Hom-bialgebra structure, see \cite{ya4}.
However it has not a Hom-Hopf algebra structure in the sense of \cite{ms2}. This is due to the fact that the antipode is not an inverse of the identity
map in the convolution product. This motivated  the authors in \cite{lmt} to modify the notion of invertibility in  Hom-algebras and introduce a new definition
for the antipode of Hom-Hopf algebras. Solving this problem, they came into axioms of Hom-groups which naturally are appearing in the structure of the group-like
elements of Hom-Hopf algebras. They also were motivated by  constructing  a Hom-Lie group integrating a Hom-Lie algebra.
Simultaneously with this paper,  the author    in  \cite{h1}  introduced and studied several fundamental notions for Hom-groups which the twisting map $\a$ is invertible.  It was shown that Hom-groups are examples of quasigroups. Furthermore Lagrange theorem for finite Hom-groups were shown. The Home-Hopf algebra structure of Hom-group Hopf algebra $\mathbb{K}G$ were introduced in \cite{h2}.\\

In this paper we investigate different aspects of Hom-groups such as  modules and homological algebra.
In Section 2, we study the basics of Hom-groups. We introduce the Hom-algebra associated to a Hom-group $G$ and we call it Hom-group algebra denoted by $\mathbb{K}G$.
It has been shown in \cite{ms1} that the commutator of a Hom-associative algebra  $A$ is a Hom-Lie algebra $\mathfrak{g}_A$.
The authors in \cite{ya4}, \cite{lmt} showed that the universal enveloping algebra of a Hom-Lie algebra is endowed it with a
Hom-Hopf algebra structure. Therefore Hom-groups are  sources of examples for  Hom-algebras, Hom-Lie algebras, and Hom-Hopf algebras as follows

$$ G\hookrightarrow \mathbb{K}G\hookrightarrow \mathfrak{g}_{\mathbb{K}G}\hookrightarrow U(\mathfrak{g}_{\mathbb{K}G}).$$
We refer the reader for more examples and fundamental notions for Hom-groups to \cite{h1}.
%Later we look into some properties of homomorphism of Hom-groups and we show that  a Hom-group homomorphism
%does not necessarily need to be  unital. If a Hom-group homomorphism is unital we prove that the image of
%an inverse element will be the inverse of the image. Furthermore in this case, the kernel of a Hom-group homomorphism is a Hom-subgroup.
In Section 3, we introduce two types of modules over Hom-groups. The first type is called dual Hom-modules.
Using inverse elements in a Hom-group $G$, we show that a left dual $G$-module can be turned in to a right dual $G$-module and vice-a-versa.
Then we introduce $G$-modules and we show that  for any left $G$-module $M$ the algebraic dual $Hom(M, \mathbb{K})$ is a
dual right $G$-module where $\mathbb{K}$ is a field.
 It is known that the group (co)homology provides an important set of tools for studying modules over a group.
 This motivates us to introduce (co)homology theories for Hom-groups to find out more about representations of Hom-groups.
 Generally introducing homological algebra for non-associative objects is a difficult task. The first attempts to introduce homological tools for
 Hom-algebras and Hom-Lie algebras were appeared in \cite{aem}, \cite{ms3}, \cite{ms4}, \cite{ya1}. The authors in \cite{hss} defined Hochschild and
 cyclic (co)homology for Hom-algebras.
 In Section 4, we introduce  Hom-group cohomology with coefficients in dual left (right) $G$-modules. The  conditions of
  $M$  in our work were also appeared in other context such as \cite{cg} where the authors used the category of Hom-modules over
Hom-algebras to obtain a monoidal category for modules over Hom-bialgebras.
A noticeable difference between homology theories of Hom-algebras introduced in \cite{hss} and the ones for Hom-groups in this paper is
that the first one needs bimodules over Hom-algebras and the second one requires  one sided modules (left or right).
We  show that the Hom-group cohomology with coefficients in a dual right $G$-module is isomorphic to Hom-group cohomology with
coefficients in the dual left $G$-module
where  the left action is given by the inverse elements. We compute 0 and 1-cocycles and we show the functoriality of Hom-group cohomology for certain coefficients.
Since any Hom-group gives the Hom-group algebra, the natural question is the relation with cohomology theories of these two different objects.
We show that Hom-group cohomology of a Hom-group $G$ with coefficients in a dual left module is isomorphic to the Hom-Hochschild cohomology of Hom-group algebra $\mathbb{K}G$ with coefficients
in the dual $\mathbb{K}G$-bimodule whose right $\mathbb{K}G$-action is trivial. Later we
introduce Hom-group homologies with coefficients in left (right) $G$-modules.
 Despite the associative case, the Hom-associativity condition leads us to use different type of
representations for cohomology and homology theories for Hom-groups. We look into similar results in the homology case.
 The  (co)homology theories for Hom-algebras in \cite{hss}, \cite{aem}, \cite{ms3} and Hom-groups in this paper, gives us the hope of solving the open problems of introducing homological tools for other non-associative objects such as Jordan algebras and alternative algebras.

%%%%%%%%%%%%%%%%%%%%%%%%%%
\bigskip

%\textbf{Acknowledgments}:

%\bigskip

%\textbf{Notations}:

%%%%%%%%%%%%%%%%%%%%%%%%%%%%%%%%%%%%%%%%%%%%%%%%%%%%%%%%%%%%%%%%%%%%%%%%%%%%%%%%%%%%%%%%%%%%%%%%%%%%%%%%%%%%%%%%%%%%%%%%%%%%%%%%%%%%%%%%%%%%%%%%%%%%%%%%%%%%%%%%%%%%%%%%%%%%%%%
%%%%%%%%%%%%%%%%%%%%%%%%%%%%%%%%%%%%%%%%%%%%%%%%%%%%%%%%%%%%%%%%%%%%%%%%%%%%%%%%%%%%%%%%%%%%%%%%%%%%%%%%%%%%%%%%%%%%%%%%%%%%%%%%%%%%%%%%%%%%%%%%%%%%%%%%%%%%%%%%%%%%%%%%%%%%%%%%
\tableofcontents
%%%%%%%%%%%%%%%%%%%%%%%%%%%%%%%

\section{Hom-groups }
Here we recall the definition of a Hom-group from \cite{lmt}.
\begin{definition}\label{def-hom}{\rm

A Hom group consists of a set G together with
a distinguished member $1$ of $G$,
a set map:  $\alpha: G\longrightarrow G$,
an operation $\mu: G\times G\longrightarrow G$, and an operation written as
$^{-1}:G\longrightarrow G$.
These pieces of structure are subject to the following axioms:\\

i) The product map $\mu: G\times G\longrightarrow G$ is satisfying the Hom-associativity property
   $$\mu(\alpha(g), \mu(h, k))= \mu(\mu(g,h), \alpha(k)).$$
   For simplicity when there is no confusion we omit the sign $\mu$.

   ii)  The map $\alpha $ is multiplicative, i.e, $\alpha(gk)=\alpha(g)\alpha(k)$.

   iii) The   element $1$ is called unit and it satisfies  the Hom-unitality condition
   $$g1=1g=\alpha(g), \quad\quad ~~~~~ \a(1)=1.$$

   iv) The map $g\longrightarrow g^{-1}$ satisfies the  anti-morphism  property $(gh)^{-1}=h^{-1} g^{-1}$.

   v) For  any $g\in G$
   there exists a natural number  $n$ satisfying the Hom-invertibility condition
   $$\alpha^n(g g^{-1})=\alpha^n(g^{-1}g)=1.$$
   The smallest such
    $n$ is called the invertibility index of $g$.
  }
\end{definition}

Since we have the anti-morphism $g\longmapsto g^{-1}$, therefore by the definition,  inverse of any element $g\in G$ is
unique although different elements may have different invertibility index.
% Furthermore
%since $G$ is not associative   by anti-morphism property of the inverse   we have
% $(((g_1g_2)g_3 ) \cdots g_n)^{-1}= g_n^{-1}( g_{n-1}^{-1}(\cdots (g_2^{-1}g_1^{-1})))$ which is  different from
%$(g_1(g_2(\cdots (g_{n-1}g_n))))^{-1}=(((g_n^{-1}g_{n-1})\cdots g_2^{-1})g_1^{-1})$.
The inverse of the unit element $1$ of a Hom-group $(G, \a)$ is itself because $\a(\mu(1,  1))=\a(1)=1$.
 For any Hom-group $(G, \a)$ we have $\a(g)^{-1}=\a(g^{-1})$. This is because if $g^{-1}$ is the unique inverse of $g$ where its  invertibility index is $k$ then
$$\a^{k-1}( \a(g) \a(g^{-1}))=\a^k(g) \a^k(g^{-1})=1.$$ So the invertibility index of $\a(g)$ is $k-1$.
If $k=1$ then the invertibility index  of each element  of $G$ is one.
Non-associativity of the product prevent us to easily define the notion of order for an element $g$. Therefore many  basics result of group theory will be affected by missing associativity condition.

\begin{example}\label{deformation of groups}
  {\rm
  Let $(G, \mu,  1)$ be any group and $\a: G\longrightarrow G$ be a group homomorphism. We define a new product $\mu_{\a}: G\times G\longrightarrow G$ given by $$\mu_{\a}(g, h)= \a(\mu(g, h))=\mu(\a(g), \a(h)).$$ Then $(G, \mu_{\a})$ is a Hom-group and we denote this by $G_{\a}$. We note that inverse of any element $g\in G_{\a}$ is also $g^{-1}$ because
  $$\a(\mu_{\a}(g, g^{-1}))= \a(g)\a(g^{-1})=\a(1)=1.$$
  The invertibility index of all elements of $G_{\a}$ are one.

   %We note that  $\a(g)^{-1}= \a(g^{-1})$ because $$\mu_{\a}(\a(g), \a(g^{-1}))= \a(\a(g)\a(g^{-1}))=\a^2(1)=1.$$
  }
\end{example}
\begin{remark}
  {\rm
  In this paper we use the general definition of Hom-groups in Definition \ref{def-hom}. However the author in \cite{h1} considered an special case when $\a$ is invertible. Therefore the invertibility axiom will change to  the one that for any $g\in G$, there exists a $g^{-1}\in G$ where
  $$gg^{-1}= g^{-1}g=1.$$
  It was shown that the inverse element $g^{-1}$ is unique and also $(gh)^{-1}= h^{-1}g^{-1}$. Therefore some of the axioms in Definition \ref{def-hom} will be obtained by Hom-associativity. See \cite{H1}.

  }
\end{remark}

\begin{definition}{\rm
    Let $\mathbb{K}$ be a field. For any Hom-group $(G, \a)$ we can define a free $k$-Hom algebra $\mathbb{K}G$ which is called Hom-group algebra.
    More precisely $\mathbb{K}G$ denotes the set of all formal expressions of the form $\sum cg$ where $c\in k$ and $g\in G$. The multiplication of $\mathbb{K}G$ is defined
    by $(cg)(c'g')= (cc')(gg')$ for all $c,c'\in k$ and $g,g'\in G$.
    For the Hom-algebra structure  we  extend $\a:G\longrightarrow G$ to a $\mathbb{K}$-linear map $\mathbb{K}G\longrightarrow \mathbb{K}G$ in the obvious way.

    }
      \end{definition}

      \begin{remark}{\rm

It is shown in \cite{ms1} that the commutator of a Hom-associative algebra  $A$ given by  $[a,b]=ab-ba$,  is a Hom-Lie algebra $\mathfrak{g}_A$.
Furthermore the authors in \cite{ya4},  \cite{lmt} showed that the universal enveloping algebra of a Hom-Lie algebra is endowed it with a
Hom-Hopf algebra structure. One notes that  the Hopf algebra structures in \cite{ya4} is different from the one in \cite{lmt}. Therefore Hom-groups
are a  source of examples of Hom-algebras, Hom-Lie algebras, and Hom-Hopf algebras   as follows

$$ G\hookrightarrow \mathbb{K}G\hookrightarrow \mathfrak{g}_{\mathbb{K}G}\hookrightarrow U(\mathfrak{g}_{\mathbb{K}G}).$$
}
      \end{remark}

  By \cite{lmt}, an element $x$ in  an unital Hom-algebra $(A, \a, 1)$ is invertible if there is an element $x^{-1}\in A$ and a
   non-negative integer $k$ such that $$ \a^k(x x^{-1})= \a^k(x^{-1}x)=1.$$
  The element $x^{-1}$ is called the Hom-inverse of $x$.
  The Hom-inverse  of an element in a Hom-algebra may not be unique. This  is different from
  Hom-groups where the inverse of an element is unique. This prevents Hom-invertible elements in an Hom-algebra to be a Hom-group in general.
  The authors in \cite{lmt} showed that    for any unital Hom-algebra, the unit 1 is Hom-invertible, the
product of any two Hom-invertible elements is Hom-invertible and every inverse of a Hom-invertible element
is Hom-invertible. Furthermore they proved that the set of group-like elements in a Hom-Hopf algebra is  a Hom-group.
 The inverse of an element $cg$   in  the Hom-group algebra $\mathbb{K}G$   is the unique element $c^{-1}g^{-1}$, where $c\in \mathbb{K}$ and $g\in G$.
 \begin{definition}
   {\rm
   A subset $H$ of a Hom-group $(G,\a)$ is called a Hom-subgroup of $G$ if $(H,\a)$ itself is a Hom-group.

   }
 \end{definition}
 One notes that if $H$ is a Hom-subgroup of $G$ then $\a(h) = 1 h\in H$ for all $h\in H$. Therefore $\a(H)\subseteq H$.
 \begin{example}
   {\rm
   Let $G$ be a group and $\a: G\longrightarrow G$ be a group homomorphism. If $H$ is a subgroup of
   $G$ which $\a(H)\subseteq H$ then $(H_{\a}, \a)$ is a Hom-subgroup of $G_{\a}$.

   }
 \end{example}
\begin{definition}
  {\rm
  Let $(G, \a)$ and $(H, \b)$ be two Hom-groups. The morphism $f: G\longrightarrow H$ is called a morphism of Hom-groups if $\b(f(g))=f(\a(g))$ and $f(gk)=f(g)f(k)$ for all $g,k\in G$.
  Two Hom-groups $G$ and $H$ are called isomorphic if there exist a  bijective morphism of Hom-groups $f: G\longrightarrow H$.
  }
\end{definition}

\begin{proposition}
  Let $(G, \a)$ and $(H, \b)$ be two Hom-groups and  $f: G\longrightarrow H$ be  a morphism of Hom-groups.
  If  the invertibility index of the element $f(1_G)\in H$ is $n$ then $\b^{n+2}(f(1_G))=1_H$.

\end{proposition}
\begin{proof}
Since $f$ is multiplicative then $$f(1_G) f(1_G)= f(1_G 1_G) = f(1_G).$$
 Also $$1_H f(1_G)=\b(f(1_G))= f(\a(1_G))= f(1_G).$$
 Therefore $$f(1_G) f(1_G)= 1_H f(1_G).$$
 Then $\b^n(f(1_G)) \b^n(f(1_G))= \b^n(1) \b^n(f(1_G))$.
 So we have $\b^n(f(1_G)) \b^n(f(1_G))= 1_H \b^n(f(1_G))$. Then $$[\b^n(f(1_G)) \b^n(f(1_G))] \b^{n+1}(f(1_G)^{-1})= [1_H \b^n(f(1_G))] \b^{n+1}(f(1_G)^{-1}).$$
 So
 $$  \b^{n+1}(f(1_G)) [   \b^n(f(1_G) f(1_G)^{-1})] = \b(1_H) [\b^n(f(1_G) f(1_G)^{-1})]   .$$ Then

 $$\b^{n+1}(f(1_G)) 1_H = b(1_H) 1_H.    $$ Therefore $\b^{n+2}(f(1_G))= \b^2(1_H)=  1_H.$

\end{proof}
This Lemma shows that in general  for a  Hom-group homomorphism    $f: G\longrightarrow H$ the unitality condition  $f(1)=1$ does not hold.

\begin{lemma}
   Let $(G, \a)$ and $(H, \b)$ be two Hom-groups and  $f: G\longrightarrow H$ be  a morphism of Hom-groups. If $f(1_G)=1_H$ then
   $f(g^{-1})= f(g)^{-1}$.

\end{lemma}
\begin{proof}
  We suppose that the  invertibility index of $g$ is $n$. Therefore $$ \b^n(f(g) f(g^{-1}))=f(\a^n(g g^{-1}))= f(1_G)=1_H.$$
  So  $f(g^{-1})= f(g)^{-1}$.
\end{proof}

\begin{lemma}
   Let $(G, \a)$ and $(H, \b)$ be two Hom-groups and  $f: G\longrightarrow H$ be  a morphism of Hom-groups. If $f(1_G)=1_H$ then
   $kerf=\{ g\in G, ~~~ f(g)=1_H\}$ is a Hom-subgroup of $G$.
\end{lemma}
\begin{proof}
  Since $f$ is multiplicative then $kerf$ is closed under multiplication. Also if $x\in kerf$ then $x^{-1}\in kerf$ because by previous lemma
  $$f(x^{-1})= f(x)^{-1}= 1_H^{-1}=1_H.$$
\end{proof}

%%%%%%%%%%%%%%%%%%%%%%%%%%%%%%%%%%%%%%%%%%%%%%%%%%%%%%%%%%%%%%%%%%%%%%%%%%%%%%%%%%%%%%%

\section{$G$-modules }

In this section we introduce two different types  of modules over a Hom-group $G$. The first type is called dual $G$-modules and we use them
later to introduce a cohomology theory for Hom-groups. The other type is called $G$-modules and they will be used to define a homology theory of Hom-groups.

\begin{definition}
  Let $(G, \a)$ be a Hom-group. An abelian  group $M$ is called a dual left $G$-module if  there are linear maps $\cdot: G\times M\longrightarrow M$, and  $\b: M\longrightarrow M$ where
  \begin{equation}\label{left-dual-module}
    g\cdot (\a(h)\cdot m)=\b((gh)\cdot m),     \quad\quad g,h\in G,
  \end{equation}
  and $$1\cdot m= \b(m).$$

Similarly, $M$ is called a dual right $G$-module if
\begin{equation}\label{right-dual-module}
  (m\cdot \a(h))\cdot g= \b(m\cdot (hg)), \quad\quad m\cdot 1=\b(m).
\end{equation}

Finally, we call $M$ a dual $G$-bimodule if it is both a dual left and a dual right $G$-module with the following bimodule property
$$\a(a)\cdot (v\cdot b)=(a\cdot v)\cdot \a(b).$$
\end{definition}

\begin{lemma}\label{result of left dual module}
  If $ (G, \a)$ is a Hom-group  and  $M$ a  dual left $G$-module, then
  \begin{equation}
    g\cdot \b(m)= \b(\a(g)\cdot m), \quad\quad g\in G, m\in M.
  \end{equation}
  Similarly for a dual right $G$-module we have
  \begin{equation}
    \b(m\cdot \a(g))= \b(m)\cdot g.
  \end{equation}
\end{lemma}
\begin{proof}
  This is followed by substituting $h=1$ in \eqref{left-dual-module} and  \eqref{right-dual-module} and  using $1\cdot m=\b(m)= m\cdot 1$.
  \end{proof}
  It is known that for every group $G$, a  right $G$-module $M$ can be turned in to a left $G$-module $\widetilde{M}=M$
  where the left action is given by $g\cdot m:= mg^{-1}$. This process can also  be done for Hom-groups as follows.

  \begin{lemma}\label{right to left}
    Let $(G, \a)$ be a Hom-group. A  dual right $G$-module $M$ can be turned in to a dual left $G$-module $\widetilde{M}=M$ by the left action given by
    \begin{equation}
    g\cdot m:= m\cdot g^{-1}.
    \end{equation}
  \end{lemma}
  \begin{proof}

   This is followed by
   \begin{align*}
     g\cdot (\a(k)\cdot m)&= g\cdot (m\cdot \a(k)^{-1})\\
     &=  (m\cdot \a(k)^{-1})\cdot g^{-1}= (m\cdot \a(k^{-1}))\cdot g^{-1}\\
     &= \b(m\cdot (k^{-1}g^{-1}))=\b(m\cdot (gk)^{-1}) =\b((gk)\cdot m),\\
   \end{align*}
and also
\begin{equation*}
  1\cdot m=m\cdot 1^{-1}=m\cdot 1=\b(m).
\end{equation*}
  \end{proof}

  The following notion of modules over Hom-groups  will be used to introduce Hom-group homology.

  \begin{definition}\label{modules over Hom-groups}
Let $(G, \a)$ be a Hom-group. An abelian group  $V$ equipped with $\cdot :M \times V \longrightarrow V$, $a\times  v\mapsto a\cdot  v$, and $\b:V\longrightarrow V$, is called a left $G$-module if

\begin{equation}\label{aux-Hom-module}
(gk)\cdot \b(m) = \a(g)\cdot (k\cdot m), \quad\quad~~~~~~~~ 1\cdot m=\b(m),
\end{equation}
for all $g,k\in G$ and  $m\in M$.

Similarly, $(M,\b)$ is called a right $G$-module if
\begin{equation*}
\beta(m)\cdot(gk)= (m\cdot g)\cdot \a(k), \quad\quad ~~~~ m\cdot 1= \b(m).
\end{equation*}
Furthermore  $M$ is called an $G$-bimodule if
\begin{equation}\label{aux-A-bimodule}
\a(g)\cdot (m \cdot k) = (g \cdot m)\cdot \a(k),
\end{equation}
for all $g,k\in G$, and  $m \in M$.
\end{definition}

\begin{example}\rm{
  For a Hom-group $G$, the Hom-group algebra $\mathbb{K}G$ is  a bimodule  over $G$ by the left and right actions defined by its multiplication and $\b=\a$.
  More precise the left action is defined to be $g\cdot (ch)= c(gh)$ where $g,h\in G$ and $c\in k$.
  }
\end{example}

 \begin{lemma}\label{right to left-2}
    Let $(G, \a)$ be a Hom-group. A  right $G$-module $M$ can be turned in to a  left $G$-module $\widetilde{M}=M$ by the left action
    \begin{equation}
    g\cdot m:= m\cdot g^{-1}.
    \end{equation}
  \end{lemma}
\begin{proof} This is because
  \begin{align*}
    &(gk)\cdot\b(m)=\b(m) \cdot (k^{-1}g^{-1})= (m\cdot k^{-1})\cdot \a(g^{-1})\\
    &=  (m\cdot k^{-1})\cdot \a(g)^{-1}= \a(g)\cdot (m\cdot k^{-1})       =\a(g)\cdot (k\cdot m)
  \end{align*}
\end{proof}
\begin{example}\rm{
  Let $G$ be a Hom-group and $M$ be a right $G$-module.  If $\mathbb{K}$ is a field, then the  algebraic dual
  ${{M}}^*= \Hom(M,\mathbb{K})$ can be turned in to a left dual $G$-module by the left dual action given by
\begin{equation}
  (g\cdot f)(m)= f(m\cdot g).
\end{equation}

}
\end{example}
%%%%%%%%%%%%%%%%%%%%%%%%%%%%%%%%%%%%%%%%%%%%%%%%%%%%%%%%%%%%%%%%%%%%%%%%%%%%%%%%%%%%%%%%%%%%%%%%%%%%%%%%%%%%%%%

\section{Hom-group cohomoloy}

In this section we introduce Hom-group cohomology for Hom-groups. To to this we need to use the dual modules for  the proper coefficients.

\begin{theorem}\label{Hom-group cohomology for left}
  Let $(G,   \alpha)$ be a Hom- group and $(M, \beta)$ be a dual left $G$-module. Let $C^n_{Hom}(G, M)$ be the space of all  maps $\varphi: G^{\times n}\longrightarrow M$. Then
\begin{equation*}
C_{Hom}^\ast(G, M)=\bigoplus_{n\geq0} C_{Hom}^n(G, M),
\end{equation*}
with the coface maps
\begin{align}\label{aux-cosimplisial-structure-vp}
\begin{split}
&\d_0\varphi(g_1, \cdots ,  g_{n+1})=g_1\cdot \varphi(\a(g_2), \cdots , \a(g_{n+1}))\\
&\d_i\varphi(g_1 , \cdots ,g_{n+1})=\b(\varphi(\a(g_1), \cdots , g_i g_{i+1}, \cdots , \a(g_{n+1}))), ~~ 1\leq i \leq n\\
&\d_{n+1}\varphi(g_1, \cdots , g_{n+1})= \b(\varphi(\a(g_1), \cdots , \a(g_{n}))),\\
\end{split}
\end{align}
is a cosimplicial module.
\end{theorem}
\begin{proof}
We need to show that $\delta_i \delta_j= \delta_j \delta_{i-1}$ for $0\leq j< i \leq n-1$. Let us first show that $\d_1\d_0=\d_0\d_0$.
\begin{align*}
\d_0(\d_0\varphi)(g_1,\cdots , g_{n+2}) &=g_1\cdot \d_0\varphi(\a(g_2),\cdots ,\a(g_{n+2}))\\
&= g_1\cdot (\a(g_2)\cdot\varphi(\a^2(g_3),\cdots ,\a^2(g_{n+2}))) \\
& =\b((g_1g_2)\cdot\vp(\a^2(g_3),\cdots ,\a^2(g_{n+2}))) \\
&= \b(\d_0\vp(g_1g_2\, \a(g_3),\cdots ,\a(g_{n+2})))\\
&=\d_1\d_0\vp(g_1,\cdots ,g_{n+2}).
\end{align*}
We used the  left dual module property   in the third equality. Now we  show that $\delta_{n+1} \delta_n= \delta_n \delta_{n}$.
\begin{align*}
  \delta_{n+1} \delta_n\varphi(g_1,\cdots, g_{n+1})  &=\b(\delta_n\varphi(\a(g_1), \cdots, \a(g_{n})))\\
  &=\b^2(\varphi(\a(g_1), \cdots, \a(g_{n-1})))\\
  &=\b(\delta_n\varphi(\a(g_1),\cdots,\a( g_{n-1}), \a(g_{n}g_{n+1})))\\
  &=\b(\delta_n\varphi(\a(g_1),\cdots,\a( g_{n-1}), \a(g_{n})\a(g_{n+1})))\\
  &=\delta_n \delta_{n}\varphi(g_1,\cdots, g_{n+1}).
\end{align*}
We used the multiplicity of $\a$ in the fourth equality.
The following demonstrates that $\d_{n+1}\d_0=\d_0\d_{n}$. We have

\begin{align*}
  \d_{n+1}\d_0\varphi(g_1,\cdots , g_{n+1})  &=\b(\d_0\varphi(\a(g_1),\cdots ,\a(g_{n})))\\
  &=\b(\a(g_1)\cdot \varphi(\a^2(g_2),\cdots ,\a^2(g_{n})))\\
  &=g_1 \cdot\b(\varphi(\a^2(g_2),\cdots ,\a^2(g_{n})) )\\
  &=g_1\cdot \d_n\varphi(\a(g_2),\cdots, \a(g_{n+1}))\\
  &=\d_0\d_n\varphi(g_1,\cdots ,g_{n+1}).
\end{align*}

We used the Lemma \ref{result of left dual module} in the third equality.
The relations $ \delta_{j+1} \delta_j= \delta_j \delta_{j}$ follows from the Hom-associativity of $G$.
\end{proof}
Similarly we have the following result.

\begin{proposition}
  Let $(G,   \alpha)$ be a Hom-group and $(M, \beta)$ be a dual right $G$-module. Let $C^n_{Hom}(G, M)$ be the space of all  maps $\varphi: G^{\times n}\longrightarrow M$. Then
\begin{equation*}
C_{Hom}^\ast(G, M)=\bigoplus_{n\geq0} C_{Hom}^n(G, M),
\end{equation*}
with the coface maps
\begin{align}\label{aux-cosimplisial-structure-vp}
\begin{split}
&\d_0\varphi(g_1, \cdots ,  g_{n+1})= \varphi(\a(g_1), \cdots , \a(g_{n}))\cdot g_{n+1}\\
&\d_i\varphi(g_1 , \cdots ,g_{n+1})=\b(\varphi(\a(g_1), \cdots , g_ig_{i+1}, \cdots , \a(g_{n+1}))), ~~ 1\leq i \leq n\\
&\d_{n+1}\varphi(g_1, \cdots , g_{n+1})= \b(\varphi(\a(g_2), \cdots , \a(g_{n+1}))),\\
\end{split}
\end{align}
is a cosimplicial module.
\end{proposition}
\begin{proof}
Here we show that  $\d_{n+1}\d_0=\d_0\d_{n}$.
\begin{align*}
  \d_{n+1}\d_0\varphi(g_1,\cdots , g_{n+1})  &=\b(\d_0\varphi(\a(g_2),\cdots , \a(g_{n+1})))\\
  &=\b(\varphi(\a^2(g_2),\cdots ,\a^2(g_{n}))\cdot\a(g_{n+1}))\\
  &= \b(\varphi(\a^2(g_2),\cdots, \a^2(g_{n})) )\cdot g_{n+1}\\
  &= \d_n\varphi(\a(g_1),\cdots ,\a(g_{n}))\cdot  g_{n+1}\\
  &=\d_0\d_n\varphi(g_1,\cdots , g_{n+1}).
\end{align*}
We used the Lemma \ref{result of left dual module} in the third equality.
  The rest of the relations can be proved similar to the Theorem \ref{Hom-group cohomology for left}.

\end{proof}
Now we define the coboundary $b= \sum_{i=0}^{n} d_i$. The previous Theorem and Proposition imply $b^2=0$.
The  cohomology of the cochain complex
$$
\begin{CD}
0 @>b>> C_{Hom}^0(G, M) @>b>> C_{Hom}^1(G, M) @>b>> C_{Hom}^2(G, M) @>b>> C_{Hom}^3(G, M) \ldots
\end{CD}\\
$$
is called Hom-group cohomology of $G$ with coefficients in $M$. Here $M= C_{Hom}^0(G, M)$.
The following proposition shows the relation between Hom-group cohomology with coefficients with  dual left and  dual right modules.

\begin{proposition}
  Let $(G, \a)$  be a Hom-group and $M$ be a dual right $G$-module. Then  $\widetilde{M} =M$  with the left action $g\cdot m= m\cdot g^{-1}$ is a dual left $G$-module. Furthermore
  $$H^*(G, M)\cong H^*(G, \widetilde{M}).$$
  \end{proposition}
  \begin{proof}
  The space $\widetilde{M}$ is a dual left $G$-module by the Lemma \ref{right to left}.
We define
$$F: C^n(G, \widetilde{M})\longrightarrow C^n(G, M),$$ given by
$$F(\varphi)(g_1, \dots , g_n)=\varphi(g_n^{-1}, \cdots , g_1^{-1}).$$
Here we show $F \d^{\widetilde{M}}_0= \d^{M}_0 F$ where $\d^{\widetilde{M}}$ and $\d^{M}$ stand for the coface maps when
the coefficients are   $\widetilde{M} $ and $M$, respectively.
\begin{align*}
  F \d^{\widetilde{M}}_0\varphi(g_1, \cdots , g_{n+1})&=\d^{\widetilde{M}}_0\varphi(g_{n+1}^{-1}, \cdots, g_1^{-1})\\
  &=g_{n+1}^{-1}\cdot \varphi(\a(g_n^{-1}), \cdots , \a(g_1^{-1}))\\
  &=\varphi(\a(g_n^{-1}), \cdots , \a(g_1^{-1}))\cdot g_{n+1}\\
  &=\varphi(\a(g_n)^{-1}, \cdots , \a(g_1)^{-1})\cdot g_{n+1}\\
  &= F\varphi(\a(g_1), \cdots , \a(g_n))\cdot g_{n+1}\\
  &=\d^{M}_0 F(g_1, \cdots , g_{n+1}).
\end{align*}

Similarly  $F$ commutes with all $\d_i$'s and therefore with the coboundary maps $b=\sum_i \d_i$. Thus $F$ is a
 map of cochain complexes and induces a map on the level of cohomology.
 Furthermore $F$ is a bijection on the level of cochain complexes because  inverse elements are unique in Hom-groups.
  \end{proof}

The following two examples show that the cohomology classes could contain important information about  a Hom-group.
\begin{example}{\rm \textbf{({$H^0$} and twisted invariant elements) }\\
  Let $(G, \a)$ be a Hom-group and $M$ be a dual right $G$-module.
  Then $$H^0(G, M)= \{ m\in M, mg=\b(m), \forall g\in G\}.$$
  So the zero cohomology class  is the subspace  of $M$ which contains those elements that are invariant  under the $G$-action with respect to $\b$.

  }
\end{example}
\begin{example}
  {\rm \textbf{($H^1$ and twisted crossed homomorphisms )}\\

  Let $(G, \a)$ be a Hom-group and $M$ be a dual right $G$-module. To compute $H^1(G, M)$ we need to  compute $ker b$ which
  contains   the 1-cochains $f: G\longrightarrow M$  with $df(g, h)=0$. This means
 $$ f(\a(g))\cdot h -\b(f(gh))+\b(f(\a(h)))=0,$$ or
 $$\b(f(gh))= f(\a(g))\cdot h +\b(f(\a(h))).$$
These maps are called twisted crossed homomorphism of $G$. Also  $Im b$ contains all $ \varphi: G \longrightarrow M$ where there exists $m\in M$ such that
 $\varphi(g)=mg - \b(m)$.  These map are called twisted principal crossed homomorphisms of $G$.  Therefore the first cohomology is the quotient of
 twisted crossed homomorphism by twisted principal crossed homomorphisms.
  }
\end{example}

\begin{example}{\rm
  We recall that  for a Hom-group $G$ the Hom-group algebra $V=\mathbb{K}G$ is  a $G$-bimodule by multiplication of $G$.
  Therefore by examples of the previous section $(\mathbb{K}G)^\ast$ is a $G$-dual bimodule.
  Now we consider the Hom-group cohomology of $G$ with coefficients in the dual $G$-bimodule $(\mathbb{K}G)^\ast$.
  We show that the coboundary map can be written differently in this case.
One Identifies  $\varphi\in C^n(G,G^\ast)$ with
\begin{equation*}\label{aux-identification}
\phi:G^{\times\,n+1}\longrightarrow k,\qquad \phi(g_0, g_1, \cdots g_n):=\varphi(g_1 \odots g_n)(g_0).
\end{equation*}
 As a result the coboundary map will be changed in  to
\begin{align*}%\label{aux-Hoch-cobound-phi}
b:C^n(G,G^\ast)&\longrightarrow C^{n+1}(G,G^\ast),\\
b\phi(g_0, \cdots,  g_{n+1})&=\phi(g_0g_1, \a(g_2) , \cdots , \a(g_{n+1}))\\
&\quad+\sum_{j=1}^n (-1)^j\phi(\a(g_0), \cdots, g_jg_{j+1}, \cdots , \a(g_{n+1}))\\
&\quad+(-1)^{n+1} \phi(g_{n+1}g_0\ot \a(g_1), \cdots, \a(g_n)).
\end{align*}
Also  the cosimplicial structure is translated into
\begin{align}\label{aux-cosimplisial-structure-phi}
\begin{split}
    &\d_0\phi(g_0 , \cdots , g_n)= \phi(g_0g_1, \a(g_2) , \cdots, \a(g_{n}))\\
    &\d_i\phi(g_0 , \cdots , g_n)=\phi(\a(g_0) , \cdots , g_i g_{i+1} , \cdots , \a(g_{n})), ~~ 1\leq i \leq n-1\\
    &\d_{n}\phi(g_0 , \cdots , g_n)=\phi(g_{n}g_0, \a(g_1), \cdots ,\a(g_n)).
\end{split}
\end{align}
}
\end{example}

%\begin{definition}
%  {\rm
%  Let $M$ and $N$ be modules over Hom- groups $G$ and $H$. A morphism $\varphi: M\longrightarrow N$ is called a morphism of $G-H$-modules if
 % $\varphi (g\cdot m)= $

%  }
%\end{definition}
The following proposition shows the functoriality  of Hom-group cohomology with certain coefficients.
\begin{proposition}
 Let $(G, \a_{G})$ and $(G', \a_{G'})$ be two Hom-groups. Then any  morphism $f: G\longrightarrow G'$ of Hom-groups induces the map
\begin{equation*}
\widehat{f}: H_{Hom}^n(G', (\mathbb{K}G')^\ast)\longrightarrow H_{Hom}^n(G, ({\mathbb{K}G})^\ast)
\end{equation*}
% Furthermore if $G$ and $G'$ are isomorphic as Hom-groups then $H_{Hom}^n(G, (\mathbb{K}G)^\ast)\cong H_{Hom}^n(G', ({\mathbb{K}G'})^\ast)$.
\end{proposition}
\begin{proof}
  We define $F: C^n(G',\mathbb{K}G'^\ast)\longrightarrow C^n(G,\mathbb{K}G^\ast)$ given by
  $$F\varphi(g_0, \cdots, g_n)= \varphi (f(g_0), \cdots, f(g_n))$$
  The map $F$ commutes with all differentials $\delta_i$ in \eqref{aux-cosimplisial-structure-phi} because $f(\a_G(g))=\a_{G'}(f(g))$ and $ f(gk)=f(g) f(k)$. Here we only show that $F$ commutes with $\d_0$ and we leave the other commutativity relations to the reader.
  \begin{align*}
    &\d_0^GF\varphi (g_0, \cdots , g_n)\\
    &= F\varphi(g_0g_1, \alpha_G(g_2), \cdots , \a_G(g_n))\\
    &=\varphi(f(g_0g_1), f(\a_G(g_2)), \cdots , f(\a_G(g_n)))\\
    &=\varphi(f(g_0) f(g_1), \a_{G'}(f(g_2)), \cdots, \a_{G'}(f(g_n)))\\
    &=\d_0^{G'}(f(g_0), \cdots , f(g_n))\\
    & =F\d_0^{G'}\varphi (g_0, \cdots , g_n).
  \end{align*}
\end{proof}
One notes that even in the case of associative groups, $H^*(G, \mathbb{K}G)$ has not similar functoriality property.
This reminds us that  the coefficients $(\mathbb{K}G)^*$ as dual $G$-module have an important rule in  Hom-group cohomology.

\begin{example}{\rm \textbf{(Trace $0$-cocycles})
     Let $(G,  \a)$ be a Hom-group. Using the differentials in \eqref{aux-cosimplisial-structure-phi} we have
\begin{equation*}
  H_{Hom}^{0}(G, \mathbb{K}G^*)=\{ \varphi: G\longrightarrow k, \quad\quad \varphi(gh)=\varphi(hg)\}.
\end{equation*}
More precisely,   $0$-cocycles  of $G$ are trace maps on $G$.
}
\end{example}

Here we aim to find out the relation between Hom-group cohomology of a Hom-group $G$ and the Hochschild cohomology of the  Hom-group algebra $\mathbb{K}G$.
 For this we recall the Hochschild cohomology of Hom-algebras introduced in \cite{hss}. First we need to recall the definition of dual modules for Hom-algebras from \cite{hss}.
  Let $(\mathcal{A}, \a)$ be a Hom-algebra. A vector space $V$ is called a dual left $\mathcal{A}$-module if there are linear maps $\cdot: \mathcal{A}\ot V\longrightarrow V$, and  $\b: V\longrightarrow V$ where
  \begin{equation}
    a\cdot (\a(b)\cdot v)=\b((ab)\cdot v).
  \end{equation}
Similarly, $V$ is called a dual right $\mathcal{A}$-module if $v\cdot (\a(a))\cdot b= \b(v\cdot (ab))$. Finally, we call $V$ a dual $\mathcal{A}$-bimodule if $\a(a)\cdot (v\cdot b)=(a\cdot v)\cdot \a(b).$
Let   $(\mathcal{A},  \alpha)$ be  a Hom-algebra,  $(M, \beta)$  be a dual $\mathcal{A}$-bimodule and  $C^n(\mathcal{A}, M)$ be the space of all $k$-linear maps $\varphi: \mathcal{A}^{\ot n}\longrightarrow M$. Then the authors in \cite{hss} showed that
\begin{equation*}
C^\ast(\mathcal{A}, M)=\bigoplus_{n\geq0} C^n(\mathcal{A}, M),
\end{equation*}
with the coface maps
\begin{align}\label{aux-cosimplisial-structure-vp}
\begin{split}
&d_0\varphi(a_1\odots a_{n+1})=a_1\cdot \varphi(\a(a_2)\odots \a(a_{n+1}))\\
&d_i\varphi(a_1\odots a_{n+1})=\b(\varphi(\a(a_1)\odots a_i a_{i+1}\odots \a(a_{n+1}))), ~~ 1\leq i \leq n\\
&d_{n+1}\varphi(a_1\odots a_{n+1})= \varphi(\a(a_1)\odots \a(a_{n}))\cdot a_{n+1}.\\
\end{split}
\end{align}
is a cosimplicial module. The cohomology of the complex $(C^\ast(\mathcal{A}, M),b)$, where $b=\sum_{i=0}^{n+1}d_i$, is  the Hochschild cohomology of the Hom-algebra $\mathcal{A}$ with coefficients in  $M$, and is denoted by $H^\ast(\mathcal{A}, M)$.

The following theorem shows that  if the dual left $G$-module
 $M$ satisfies an extra condition $\a(a)\cdot \b(m)= \b(a\cdot m)$, $a\in \mathbb{K}G, m\in M$, then
  the group cohomology of a Hom-group $G$ with coefficients in   $M$ will reduce to Hochschild cohomology of
  the Hom-group algebra $\mathbb{K}G$ with coefficients in the dual $\mathbb{K}G$-bimodule $\widetilde{M}=M$ where the left action
  is coming from the left action of $G$ and the  right action is trivial. One notes  that if $\a=\b=\Id$ then we
   obtain the corresponding well-known result in  the associative case.% which says that the group cohomology of a group $G$ with coefficients in a  left module $M$ reduces to the Hochschild cohomology of  the group algebra $\mathbb{K}G$ with coefficients in a bimodule $\widetilde{M}=M$ with a trivial right action.

  \begin{theorem}\label{homology-relations}

  Let $(G, \a)$ be a Hom-group and $M$ a dual left $G$-module.
  If  $\a(a)\cdot \b(m)= \b(a\cdot m)$, then
    $\widetilde{M}=M$ is a dual $\mathbb{K}G$-bimodule where the left action is coming from the
    original left action of $G$  and    the right action is the trivial action $m\cdot g := m\cdot 1 = \b(m)$. Furthermore

    $$H^*(G, M)\cong H^\ast(\mathbb{K}G,\widetilde{M}).$$
  \end{theorem}

  \begin{proof}
The condition $\a(a)\cdot \b(m)= \b(a\cdot m)$ insures that $\widetilde{M}$ with the given dual left action and the right trivial action $m\cdot g = m\cdot 1 = \b(m)$ is a dual $G$-bimodule, and therefore a dual $\mathbb{K}G$-bimodule, because
$$\a(g)\cdot (m\cdot k)= \a(g)\cdot \b(m)= \b(g\cdot m)=(g\cdot m)\cdot \a(k).$$
Now all differentials $d_i$ of Hochschild cohomology of $\mathbb{K}G$ will be the same as the ones, $\d_i$, for group cohomology of $G$. Therefore the identity map $\Id: C^n(G, M)\longrightarrow C^n(\mathbb{K}G, \widetilde{M})$ induces an isomorphism on the level of complexes.
  \end{proof}

 One knows that if $G$ is an  associative group, and $\mathbb{K}G$ the group algebra, then any  $\mathbb{K}G$-bimodule $M$  can be turned in to a
$G$-right (or left) module by the adjoint action. The process of having the similar
result for  Hom-groups is not clear specially because we do not know how to define the adjoint action for Hom-groups.

%%%%%%%%%%%%%%%%%%%%%%%%%%%%%%%%%%%%%%%%%%%%%%%%%%%%%%%%%

\section{Hom-group homology}

In this section we introduce homology theory for Hom-groups. To do this we need to use the notion of modules instead of dual modules for the coefficients.

\begin{theorem}
   Let $(G,   \alpha)$ be a Hom-group and $(M, \beta)$ be a right
    $G$-module satisfying
    $$\b(m\cdot g)= \b(m)\cdot \a(g).$$
     Let $C_n^{Hom}(G, M)= M\times G^{\times n}$. Then
\begin{equation*}
C_\ast^{Hom}(G, M)=\bigoplus_{n\geq0} C_n^{Hom}(G, M),
\end{equation*}
with the face maps
\begin{align}\label{aux-cosimplisial-structure-vp}
\begin{split}
&d_0(m, g_1, \cdots ,  g_{n})=(m\cdot g_1, \a(g_2), \cdots , \a(g_{n}))\\
&d_i(m, g_1 , \cdots ,g_{n+1})=(\b(m), \a(g_1), \cdots , g_i g_{i+1}, \cdots , \a(g_{n})), ~~ 1\leq i \leq n-1\\
&d_{n}(m, g_1, \cdots , g_{n})= (\b(m), \a(g_1), \cdots , \a(g_{n-1})),\\
\end{split}
\end{align}
is a simplicial module.
\end{theorem}
\begin{proof}
First we show $d_0d_0=d_0d_1$.
  \begin{align*}
    d_0d_0(m, g_1, \cdots ,  g_{n})&=d_0(m\cdot g_1, \a(g_2), \cdots , \a(g_{n}))\\
    &=((m\cdot g_1)\cdot \a(g_2), \a^2(g_3), \cdots , \a^2(g_{n}))\\
    &= ((m\cdot g_1)\cdot \a(g_2), \a^2(g_3), \cdots , \a^2(g_{n}))\\
     &= (\b(m)\cdot (g_1g_2), \a^2(g_3), \cdots , \a^2(g_{n}))\\
     &= d_0(\b(m), g_1g_2, \a(g_3), \cdots , \a(g_{n}))\\
     &=d_0d_1(m, g_1, \cdots ,  g_{n}).
  \end{align*}
  We used the dual right property in the fourth equality. Here we show $d_0 d_j=d_{j-1}d_0$ for $j> 1$.
  \begin{align*}
    d_0d_j(m, g_1,&\cdots, g_n)\\
    &=d_0(\b(v)\ot \a(a_1)\odots a_j a_{j+1}\odots \a(a_n)) \\
    &=\b(v)\cdot\a(a_1)\ot \a^2(a_2)\odots \a(a_j a_{j+1})\odots  \a^2(a_n) \\
    &=\b(v\cdot a_1)\ot \a^2(a_2)\odots \a(a_j )\a(a_{j+1})\odots  \a^2(a_n) \\
    &=d_{j-1}(v\cdot a_1\ot \a(a_2)\odots \a(a_n)) \\
    &=d_{j-1}\d_0(m\ot a_1\odots a_n).
  \end{align*}
  We used the condition $\b(m\cdot g)= \b(m)\a(g)$  in the third equality.   Now we show $d_0d_n=d_{n-1}d_0$. We have,
  \begin{align*}
    d_0d_n(m, g_1&, \cdots, g_n)\\
    &=d_0(\b(m), \a(g_1)\cdots , \a(g_{n-1}))\\
    &=(\b(m)\cdot \a(g_1),  \a^2(g_2), \cdots \a^2(g_{n-1}))\\
    &=\b(m\cdot g_1),  \a^2(g_2)\cdots \a^2(g_{n-1})\\
    &=d_{n-1}(m\cdot g_1, \a(g_2)\cdots, \a(g_n))\\
    &=d_{n-1}\d_0(m, g_1\cdots g_n).
  \end{align*}
  Now we show $d_i d_n=d_{n-1}d_i$.
  \begin{align*}
    d_i d_n(m, g_1&, \cdots ,g_n)\\
    &=d_i (\b(m), \a(g_1), \cdots, \a(g_{n-1}))\\
    &= (\b^2(m), \a^2(g_1), \cdots, \a(g_i)\a(g_{i+1}), \cdots , \a^2(g_{n-1}))\\
    &= d_{n-1}(\b(m), \a(g_1), \cdots, g_i g_{i+1}, \cdots, \a(g_n))\\
    &=d_{n-1}d_i(m, g_1, \cdots ,g_n).
  \end{align*}
  The rest of the commutativity relations can also be verified.
\end{proof}

We define the boundary map $b= \sum_{i=0}^{n} d_i$. By the previous Theorem we have $b^2=0$.
The homology of the chain complex
$$
\begin{CD}
0 @<b<< M= C^{Hom}_0(G, M) @<b<< C^{Hom}_1(G, M) @<b<< C^{Hom}_2(G, M) @<b<< C^{Hom}_3(G, M) \ldots
\end{CD}\\
$$
is called Hom-group homology of $G$ with coefficients in $M$.
Similarly one has Hom-group homology with coefficients in left modules as follows.

\begin{proposition}

   Let $(G,   \alpha)$ be a Hom-group and $(M, \beta)$ be a left
    $G$-module satisfying
    $$\b(g\cdot m)= \a(g)\cdot \b(m)  .$$
     Let $C_n^{Hom}(G, M)=G^{\times n}$. Then
\begin{equation*}
C_\ast^{Hom}(G, M)=\bigoplus_{n\geq0} C_n^{Hom}(G, M),
\end{equation*}
with the face maps
\begin{align}\label{aux-cosimplisial-structure-vp}
\begin{split}
&d_0( g_1, \cdots ,  g_{n}, m)=(\a( g_1), \a(g_2), \cdots ,\a( g_{n-1}), g_n\cdot m)\\
&d_i(g_1 , \cdots ,g_{n+1}, m)=( \a(g_1), \cdots , g_i g_{i+1}, \cdots , \a(g_{n}),  \b(m)), ~~ 1\leq i \leq n-1\\
&d_{n}(g_1, \cdots , g_{n}, m)= ( \a(g_2), \cdots , \a(g_{n}), \b(m)).\\
\end{split}
\end{align}
is a simplicial module.

\end{proposition}
\begin{proof}
  Similar as the  previous theorem.
\end{proof}
Here we state the relation between Hom-group homology with coefficients in left and right modules.

\begin{proposition}
  Let $(G, \a)$  be a Hom-group and $M$ be a  right $G$-module.
  Then   $\widetilde{M} =M$  with the left action $g\cdot m= m\cdot g^{-1}$, is a left $G$-module and
  $$H_*(G, M)\cong H_*(G, \widetilde{M}).$$
  \end{proposition}
  \begin{proof}
 The space $\widetilde{M}$ is a left module by the Lemma \ref{right to left-2}.
We set
$$F: C_n(G, \widetilde{M})\longrightarrow C_n(G, M),$$ given by
$$F(g_1, \dots , g_n, m)=(m, g_n^{-1}, \cdots , g_1^{-1}).$$
Here we show  $d_0^M F= F d_0^{\widetilde{M}}$.
\begin{align*}
  &d_0^M F( g_1, \cdots ,  g_{n}, m)\\
  &=d_0^M(m, g_n^{-1}, \cdots , g_1^{-1})\\
  &=(m\cdot g_n^{-1}, \a(g_{n-1}^{-1}), \cdots , \a(g_1^{-1}))\\
   &=(m\cdot g_n^{-1}, \a(g_{n-1})^{-1}, \cdots , \a(g_1)^{-1})\\
   &=F( \a(g_1), \cdots , \a(g_{n-1}), m\cdot g_n^{-1}))\\
  &=F( \a(g_1), \cdots , \a(g_{n-1}), g_n\cdot m))\\
  &=F d_0^{\widetilde{M}}( g_1, \cdots ,  g_{n}, m).
\end{align*}
Now we prove
 $d_n^M F= F d_n^{\widetilde{M}}$.
\begin{align*}
  &d_n^M F( g_1, \cdots ,  g_{n}, m)\\
    &=  d_n^M(m, g_n^{-1}, \cdots , g_1^{-1})\\
    &= (\b(m), \a(g_n^{-1}), \cdots , \a(g_2^{-1}))\\
    &=(\b(m), \a(g_n)^{-1}, \cdots , \a(g_2)^{-1})\\
    &=F(\a( g_2), \cdots ,  \a(g_{n}), \b(m))\\
    &= F d_n^{\widetilde{M}}( g_1, \cdots ,  g_{n}, m).
\end{align*}
  \end{proof}

Here we show that the Hom-group homology of a Hom-group with coefficients in a right (left) module
 reduces to Hochschild homology of Hom-group algebra with coefficients in a certain bimodule.
 To do this we remind that the  authors in \cite{hss} introduce Hochschild homology of a Hom-algebra $A$ as follows.
Let $(A,\mu, \alpha)$ be a Hom-algebra, and $(V, \beta)$ be an $A$-bimodule \cite{hss} such that
$$\b(v\cdot a) = \b(v)\cdot \a(a) \quad \text{and} \quad \b(a\cdot v)=\a(a)\cdot \b(v).$$ Then
\begin{equation*}
C^{Hom}_\ast(A, V)=\bigoplus_{n\geq 0}C^{Hom}_n(A, V),\qquad C^{Hom}_n(A, V):=V\ot A^{\ot n},
\end{equation*}
with the face maps
\begin{align*}
&\d_0(v\ot a_1\ot \cdots \ot a_{n})= v \cdot a_1 \ot \a(a_2)\ot \cdots \ot \a(a_{n})\\
&\d_i(v\ot a_1\ot \cdots \ot a_{n})=\b(v)\ot \a(a_1) \cdots \ot a_i a_{i+1}\ot \cdots\ot \alpha(a_{n}), ~~ 1\leq i \leq n-1\\
&\d_{n}(v\ot a_1\ot \cdots \ot a_n)= a_{n} \cdot v \ot \a(a_1)\ot  \cdots\ot \a(a_{n-1}),
  \end{align*}
is a simplicial module.
Similar to the  cohomology case we have the following result.

\begin{theorem}
  Let $(G, \a)$ be a Hom-group and $M$ be a  right $G$-module  satisfying $\b(m\cdot g)=  \b(m) \cdot  \a(g)  .$
   Let  $\widetilde{M}=M$  be a left module with the trivial left action $ g \cdot m:= m\cdot 1 = \b(m)$. Then $\widetilde{M}$ will be a $\mathbb{K}G$-bimodule and furthermore
    $$H_*(G, M)\cong H_\ast(\mathbb{K}G, \widetilde{M}).$$
  \end{theorem}
\begin{proof}
The condition $\b(m\cdot g)=  \b(m) \cdot  \a(g) $ insures that $\widetilde{M}$ with the given  right action and the  trivial left  action $g \cdot m:= m\cdot 1 = \b(m)$ is a  $G$-bimodule, and therefore a  $\mathbb{K}G$-bimodule, because

$$\a(g)\cdot (m\cdot k)= \b (m\cdot k)= \b(m) \cdot\a(k)=(g\cdot m)\cdot \a(k).$$
Therefore  all differentials $\d_i$ of Hochschild homology of $\mathbb{K}G$ will be the same as the ones, $d_i$, for group homology of $G$.
Therefore the identity map $\Id: C_n(G, M)\longrightarrow C_n(\mathbb{K}G, \widetilde{M})$ induces an isomorphism on the level of complexes.
\end{proof}

The  conditions of $M$ in the previous theorem were also appeared in other context such as \cite{cg} where the authors used the category of Hom-modules over
Hom-algebras to obtain a monoidal category for modules over Hom-bialgebras.

  The  functoriality  of Hom-group homology is shown as follows.
\begin{proposition}
 Let $(G, \a_{G})$ and $(G', \a'_{G'})$ be two Hom-groups. The morphism $f: G\longrightarrow G'$ of Hom-groups induces the map
\begin{equation*}
\widehat{f}: H^{Hom}_n(G, \mathbb{K}G)\longrightarrow H^{Hom}_n(G', \mathbb{K}G')
\end{equation*}
 given by $$g_0, \cdots , g_n\mapsto f(g_0), \cdots, f(g_n).$$ %Furthermore if $G$ and $G'$ are isomorphic as Hom-groups then $H^{Hom}_n(G, G)\cong H^{Hom}_n(G', G')$.
\end{proposition}
\begin{proof}
  The map $\widehat{f}$ commutes with all faces $\delta_i$ because $f(\a(g))=\a'(f(g))$ and $ f(gk)=f(g) f(k)$.
\end{proof}
%%%%%%%%%%%%%%%%%%%%%%%%%%%%%%%%%%%%%%%%%%%%%%%%%%%%%%%%%%%%%%%%%%%%%%%%%%%%%%%%%%%%%%%%%%%%%%%%%%%%%%%

\end{document}